\documentclass{amsart}

\usepackage{amsmath,amssymb,amsfonts,enumerate}
\usepackage{dsfont,mathrsfs,euscript}
\usepackage{blkarray,float,hhline}

\usepackage{pgf,tikz,pgfplots}
\pgfplotsset{compat=1.14}
\usetikzlibrary{arrows}
\usetikzlibrary{matrix,decorations.pathreplacing}
 
\usepackage[colorlinks=true,citecolor=blue,linkcolor=red,urlcolor=blue]{hyperref}

\renewcommand{\arraystretch}{1.3}

\newtheorem{theorem}{Theorem}[section]
\newtheorem{lemma}[theorem]{Lemma}
\newtheorem{proposition}[theorem]{Proposition}
\newtheorem{corollary}[theorem]{Corollary}
\newtheorem{conjecture}[theorem]{Conjecture}
\newtheorem{problem}[theorem]{Problem}
\theoremstyle{definition}

\newtheorem{example}[theorem]{Example}

\theoremstyle{remark}
\newtheorem{remark}[theorem]{Remark}

\numberwithin{equation}{section}

\begin{document}

\title[Convex functions on graphs: Sum of the eigenvalues]{Convex functions on graphs: Sum of the eigenvalues}


\author[A. Bahmani]{Asghar Bahmani}
\address{Department of Mathematics and Computer Science, Amirkabir University of Technology, P.O. Box: 15875-4413, Tehran, Iran}
\email{asghar.bahmani@gmail.com, asghar.bahmani@aut.ac.ir}

\subjclass[2010]{Primary 05C50; Secondary 15A42, 15A18, 90C27}
\keywords{convex function; tree; graph; sum of the eigenvalues; weighted conjugate degree sequence; Laplacian matrix}%




\begin{abstract}
Let $G$ be a simple graph with the Laplacian matrix $L(G)$ and let $e(G)$ be the number of edges of $G$. A conjecture by Brouwer and a conjecture by Grone and Merris state that the sum of the $k$ largest Laplacian eigenvalues of $G$ is at most $e(G)+\binom{k+1}{2}$ and $\sum_{i=1}^{k}d_{i}^{*}$, respectively, where $(d_{i}^{*})_{i}$ is the conjugate of the degree sequence $(d_i)_{i}$. We generalize these conjectures to weighted graphs and symmetric matrices. Moreover, among other results we show that under some assumptions, concave upper bounds on convex functions of symmetric real matrices are equivalent to concave upper bounds on convex functions of $(0,1)$  matrices.
\end{abstract}

\maketitle

\section{Introduction}

In this paper, all matrices are real, unless otherwise noted. We denote by $M_{n}(\mathbb{R})$ and $Sym_{n}(\mathbb{R})$ the set of all real matrices and the symmetric real matrices of order $n$, respectively. A symmetric matrix $A$ is said to be \textit{positive semi-definite}, if $\boldsymbol{x}^{T}A\boldsymbol{x}\geq 0$ for every $\boldsymbol{x}\in\mathbb{R}^{n}$ and we write $A\geq 0$. For a positive integer $m$, we denote by $[m]$, the set $\{1,\ldots,m\}$ and we denote by $\boldsymbol{j}_{m}$ the all $1$'s vector in $\mathbb{R}^{m}$. If $S\subseteq \mathbb{R}^{m}$, we denote by $S^{\perp}$ the orthogonal complement of $S$. For a positive integer $i$, $\boldsymbol{e}_{i}$ denotes the vector with a $1$ in the $i^{\rm th}$ coordinate and $0$'s elsewhere, in $\mathbb{R}^{m}$.

For two positive integers $i,j\leq m$, $E_{ij}$ denotes the  $m\times m$ matrix $\boldsymbol{e}_{i}\boldsymbol{e}_{j}^{T}$. 
For a set $X$ and $t$ disjoint nonempty subsets $X_{1},\ldots,X_{t}$, we denote by $\dot{\cup}_{i=1}^{t}X_{i}$, a partition of $X$.
For a vector $\boldsymbol{x}=(x_{1},\ldots,x_{m})\in\mathbb{R}^{m}$ and a subset $\mathcal{P}\subseteq [m]$, we denote the restriction of $\boldsymbol{x}$ to the index set $\mathcal{P}$ by $\boldsymbol{x}_{|\mathcal{P}}$; $Diag(\boldsymbol{x})$ is a diagonal $m\times m$ matrix with diagonal entries $Diag(\boldsymbol{x})_{ii}=x_{i}$. We denote by $x_{\downarrow 1}\geq\cdots \geq x_{\downarrow m}$, the decreasing order of $(x_{i})_{i}$.

Assume that $X\subseteq \mathbb{R}^{n}$ is a vector space over $\mathbb{R}$. A function $f:X\rightarrow \mathbb{R}$ is called \textit{positively homogeneous}
if $f(\alpha\boldsymbol{v}) = \alpha f(\boldsymbol{v})$, for all positive $\alpha\in \mathbb{R}$ and $\boldsymbol{v}\in X$. For a Hermitian matrix $M$ with  eigenvalues  $\lambda_{1}(M)\geq \cdots\geq \lambda_{n}(M)$ and $k\in[n]$, $S_{k}(A)$ denotes $\sum_{j=1}^{k}\lambda_{j}(M)$.
For a symmetric matrix $A=[\omega_{ij}]$ with zero diagonal, we define $e(A)=\sum_{i<j}\omega_{ij}$.
The Laplacian matrix of a simple graph $G$, denoted by $L_{G}$, is $D(G)-A(G)$, where $D(G)$ is a diagonal matrix with diagonal entries $(d_{i})_i$ and $A(G)$ is the adjacency matrix of $G$. We denote by $S_n$ and $P_n$ the star and the path of order $n$, respectively.

Interlacing theorem is a key theorem in spectral theory and we state it in the sequel.

\begin{theorem}[Interlacing Theorem]{\rm \cite[Theorem 3.2]{zhan}}\label{int}
	Let  $A$  be a  Hermitian matrix  with  eigenvalues  $\lambda_{1}\geq \cdots\geq \lambda_{n}$   and  let  $B$  be  one  of  its
	principal  submatrices.  If  the  eigenvalues  of $B$  are  $\theta_{1}\geq \cdots\geq \theta_{m}$, then $\lambda_{i}\geq \theta_{i}\geq \lambda_{n-m+i}$, for $i\in [m]$.
\end{theorem}

The following theorems are important facts on convexity of the function $S_{k}(A)$ that we use it in this paper.

\begin{theorem}\label{trace}{\rm \cite[Lemma 3.7]{zhan}}
	Let  $A$  be an $n\times n$ Hermitian matrix, $1\leq k\leq n$. Then 
	\[
	\sum_{j=1}^{k}\lambda_{j}(A)=\max_{U^{*}U=I_{k}}tr(U^{\ast}AU),
	\]
	\[
	\sum_{j=1}^{k}\lambda_{n-j+1}(A)=\min_{U^{*}U=I_{k}}tr(U^{\ast}AU),
	\]
	where $U\in M_{n,k}(\mathbb{C})$ and $M_{n,k}(\mathbb{C})$ is the set of all $n\times k$ matrices over $\mathbb{C}$.
\end{theorem}

\begin{theorem}\label{summaj}{\rm \cite[Theorem 3.8]{zhan}}
	Let $A,B$ be $n\times n$ Hermitian matrices, $1\leq k\leq n$. Then
	\[\sum_{j=1}^{k}\lambda_{j}(A+B)\leq \sum_{j=1}^{k}\lambda_{j}(A)+\sum_{j=1}^{k}\lambda_{j}(B).\]
	For $k=n$ the inequality is an equality.
\end{theorem}
Suppose that $G$ is a graph of order $n$ and degree sequence $(d_1,\ldots,d_n)$. The conjugate sequence degree of  $(d_1,\ldots,d_n)$ is defined as $(d_1^{*},\ldots,d_n^{*})$, where $d_{j}^{*}=|\{i:\, d_i\geq j\}|$.
In this paper we consider weighted versions for the following conjectures:
\begin{conjecture}[Grone-Merris Conjecture]\cite{gm}
Let $k,n\in \mathbb{N}$ and $k\in[n]$. For any simple graph $G$, if $\mu_1(L_G)\geq\cdots\geq\mu_n(L_G)$ are the Laplacian eigenvalues of $G$, then we have $\sum_{i=1}^{k}\mu_i(L_G)\leq \sum_{i=1}^{k} d_{i}^{*}$. 
\end{conjecture}
Grone-Merris Conjecture is proved in \cite{ba}.

\begin{conjecture}[Brouwer's Conjecture]\label{brouwer}\cite{bh}
	Let $k,n\in \mathbb{N}$ and $k\in[n]$. For any simple graph $G$ of order $n$, if $\mu_1(L_G)\geq\cdots\geq\mu_n(L_G)$ are the Laplacian eigenvalues of $G$, then we have $\sum_{i=1}^{k}\mu_i(L_G)\leq e(G)+\binom{k+1}{2}$. 
\end{conjecture}
Brouwer's Conjecture is proved in some special cases: see \cite{dz}, \cite{hmt}, and \cite{whl}.

In this paper after introduction section, in Section 2, we consider convex functions with some special upper bounds. Under some assumptions we show the equivalence of these bounds in two cases: real convex domains and $(0,1)$ domains. In Sections 3 and 4, we generalize Grone-Merris Conjecture and Brouwer's Conjecture to weighted graphs, respectively. In Section 5, we consider positive semi-definite decompositions of matrices and their relation to the sum of the eigenvalues. In the last section, we state some lower and upper conjectured bounds on the sum of the largest (and smallest) eigenvalues of trees and simple graphs.

\section{Convex Functions and Concave Bounds}

Suppose that $U\subseteq \mathbb{R}^{m}$. We denote by $U_{\{0,\pm 1\}}$ ($U_{\{0,1\}}$, respectively) the set $\{(a_{1},\ldots,a_{m})\in U:a_{i}\in \{-1,0,1\}, i\in [m]\}$ ($\{(a_{1},\ldots,a_{m})\in U:a_{i}\in \{0,1\}, i\in [m]\}$, respectively). In the following theorem, for convex functions we show that the concave bounds over domains $U$ and $U_{\{0,\pm 1\}}$ are equivalent.
First, we need the following lemma that is straightforward, so we omit its proof.
\begin{lemma}\label{conpos}
	Suppose that  $\varphi:U\rightarrow \mathbb{R}$ is convex. If $\varphi$ is positive homogeneous, then $\varphi(x+y)\leq \varphi(x)+\varphi(y)$, for any $x,y\in U$. 
\end{lemma}

\begin{theorem}\label{equiv0}
Let $t,m\in \mathbb{N}$, $P_{1},\ldots,P_{t}\subseteq [m]$, and $[m]=\dot{\cup}_{i=1}^{t}P_{i}$. Suppose that $\mathcal{P}\subseteq [m]$, $U=\{\displaystyle\sum_{i\in \mathcal{P}}a_{i}\boldsymbol{e}_{i}:a_{i}\in \mathbb{R}\}$, and $\varphi:U\rightarrow \mathbb{R}$ is convex and positive homogeneous.
 If $f:U\rightarrow \mathbb{R}$ is one of the following functions: 
\begin{itemize}
\item[(i.)] A concave and positive homogeneous function,
\item[(ii.)]for given $\{\alpha_{ij}\in \mathbb{R}\}$,  $f(\boldsymbol{a})=\sum\limits_{j}\sum\limits_{i=1}^{t}\alpha_{ij}\boldsymbol{a}^{i}_{\downarrow j}$, where $\boldsymbol{a}^{i}=\boldsymbol{a}_{|P_{i}}$,
\end{itemize}
then  the following statements are equivalent: 
\begin{enumerate}
\item
For every $\boldsymbol{a}\in U$, $\varphi(\boldsymbol{a})\leq f(\boldsymbol{a})$.
\item
For every $\boldsymbol{a}\in U_{\{0,\pm 1\}}$, $\varphi(\boldsymbol{a})\leq f(\boldsymbol{a})$.
\end{enumerate}
\end{theorem} 

\begin{proof}$(1)\rightarrow (2)$:
We have $U_{\{0,\pm 1\}}\subseteq U$ and the result is done.\\
$(2)\rightarrow (1)$:
Assume that $\boldsymbol{a}\in U$ and $k$ is the number of nonzero elements of $\boldsymbol{a}$. We prove this case by strong induction on $k$. If $k=0$, then $\boldsymbol{a}\in U_{\{0,\pm 1\}}$ and we have $\varphi(\boldsymbol{a})\leq f(\boldsymbol{a})$.
So, we suppose that $k>0$ and the assertion is true for $0,\ldots,k-1$. Suppose that $\boldsymbol{a}\in U$, $\alpha=\min_{j}\{|a_{j}|:a_{j}\neq 0\}$ and $\bar{\boldsymbol{a}}=(\bar{a}_{1},\ldots,\bar{a}_{m})$ is a member of $U_{\{0,\pm 1\}}$ such that $\bar{a}_{j}=sgn(a_{j})$, $j\in [m]$. Therefore  $\boldsymbol{a}-\alpha\bar{\boldsymbol{a}}\in U$ and  $\boldsymbol{a}=(\alpha\bar{\boldsymbol{a}})+(\boldsymbol{a}-\alpha\bar{\boldsymbol{a}})$.
Hence, 
\begin{align*}
\varphi(\boldsymbol{a}) & =\varphi(\alpha\bar{\boldsymbol{a}}+(\boldsymbol{a}-\alpha\bar{\boldsymbol{a}}))&\\
& \leq \varphi(\alpha\bar{\boldsymbol{a}})+\varphi(\boldsymbol{a}-\alpha\bar{\boldsymbol{a}})& \text{by Lemma \ref{conpos},}\\
& = \alpha\varphi(\bar{\boldsymbol{a}})+\varphi(\boldsymbol{a}-\alpha\bar{\boldsymbol{a}})&\\
& \leq \alpha f(\bar{\boldsymbol{a}})+f(\boldsymbol{a}-\alpha\bar{\boldsymbol{a}})&\text{by induction and (2),}\\
& \leq f(\boldsymbol{a}).& \text{(*)}
\end{align*}
Since for any decreasing orders of  $\bar{\boldsymbol{a}}$, $\boldsymbol{a}-\alpha\bar{\boldsymbol{a}}$, and $\boldsymbol{a}$ we have $\alpha f(\bar{\boldsymbol{a}})+f(\boldsymbol{a}-\alpha\bar{\boldsymbol{a}})=f(\boldsymbol{a})$,
 thus (*) holds for the function (ii.).
\end{proof}

\begin{remark}
Let $\psi:X\rightarrow \mathbb{R}$ be a convex function. Since $-\psi$ is concave, the statement similar to theorem above, for lower convex bounds on concave functions is true.
\end{remark}

\subsection{Matrix Parameters and Matrix Entries}

A parameter $\psi$ of symmetric matrices is \textit{convex parameter} (\textit{concave parameter}, respectively), if for any matrices $A,B\in Sym_{n}(\mathbb{R})$ and $\alpha\in [0,1]$, we have $\psi(\alpha A+(1-\alpha)B)\leq \alpha\psi(A)+(1-\alpha)\psi(B)$ ($\psi(\alpha A+(1-\alpha)B)\geq \alpha\psi(A)+(1-\alpha)\psi(B)$, respectively).\\

Suppose that $V\subseteq M_{n}(\mathbb{R})$. We denote by $V_{\{0,\pm 1\}}$ ($V_{\{0,1\}}$, respectively) the set $\{[\omega_{ij}]\in V:\omega_{ij}\in \{-1,0,1\}, i,j\in [n]\}$ ($\{[\omega_{ij}]\in V:\omega_{ij}\in \{0,1\}, i,j\in [n]\}$, respectively).
For a matrix $A=[\omega_{ij}]$, we denote by $A_{\downarrow 1}\geq\cdots \geq A_{\downarrow n^{2}}$ the decreasing order of $(\omega_{ij})_{i,j}$. Similar to Theorem \ref{equiv0}, we state the following theorem for real matrices.

\begin{theorem}\label{equimat2}
Let $n\in \mathbb{N}$, $P\subseteq [n]\times [n]$, and $V=\{\displaystyle\sum_{(i,j)\in P}\omega_{ij}E_{ij}:\omega_{ij}\in\mathbb{R}\}$. Suppose that $\psi:V\rightarrow \mathbb{R}$ is a convex and positive homogeneous function. If $f:V\rightarrow \mathbb{R}$ is one of the following functions:
\begin{itemize}
\item[(i.)] A concave and positive homogeneous function,
\item[(ii.)]for given $\{\alpha_{i}\in \mathbb{R}\}$, $f(A)=\sum\limits_{i\in [n^{2}]}\alpha_{i}A_{\downarrow i}$,
\item[(iii.)]for given $\{\alpha_{ij}\in \mathbb{R}\}$, $f(A)=\sum\limits_{j\in [n]}\sum\limits_{i\in [n]}\alpha_{ij}\boldsymbol{r}^{i}_{\downarrow j}$, where $\boldsymbol{r}^{i}$ is the $i^{\rm th}$ row of $A$,
\end{itemize}
then  the following statements are equivalent: 
\begin{enumerate}
\item
For every $A\in V$,  $\psi(A)\leq f(A)$.
\item
For every $A\in V_{\{0,\pm 1\}}$,  $\psi(A)\leq f(A)$.
\end{enumerate}
\end{theorem}

Same as Theorems \ref{equiv0} and \ref{equimat2}, we state the theorems below for $U_{\{0,1\}}$.

\begin{theorem}[]\label{equilinear}
Let $t,m\in \mathbb{N}$, $P_{1},\ldots,P_{t}\subseteq [m]$, and $[m]=\dot{\cup}_{i=1}^{t}P_{i}$. Suppose that $\mathcal{P}\subseteq [m]$, $U=\{\displaystyle\sum_{i\in \mathcal{P}}a_{i}\boldsymbol{e}_{i}:a_{i}\geq 0\}$, and $\varphi:U\rightarrow \mathbb{R}$ is convex and positive homogeneous.
 If $f:U\rightarrow \mathbb{R}$ is one of the following functions:
\begin{itemize}
\item[(i.)] A concave and positive homogeneous function,
\item[(ii.)]for given $\{\alpha_{ij}\in \mathbb{R}\}$, $f(\boldsymbol{a})=\sum\limits_{j}\sum\limits_{i=1}^{t}\alpha_{ij}\boldsymbol{a}^{i}_{\downarrow j}$, where $\boldsymbol{a}^{i}=\boldsymbol{a}_{|P_{i}}$,
\end{itemize}
then  the following statements are equivalent: 
\begin{enumerate}
\item
For every $\boldsymbol{a}\in U$,  $\varphi(\boldsymbol{a})\leq f(\boldsymbol{a})$.
\item
For every $\boldsymbol{a}\in U_{\{0,1\}}$,  $\varphi(\boldsymbol{a})\leq f(\boldsymbol{a})$.
\end{enumerate}
\end{theorem}

\begin{theorem}[]\label{equimat}
Let $n\in \mathbb{N}$, $P\subseteq [n]\times [n]$, and $V=\{\displaystyle\sum_{(i,j)\in P}\omega_{ij}E_{ij}:\omega_{ij}\geq 0\}$. Suppose that $\psi:V\rightarrow \mathbb{R}$ is a convex and positive homogeneous function. If $f:V\rightarrow \mathbb{R}$ is one of the following functions:
\begin{itemize}
\item[(i.)] A concave and positive homogeneous function,
\item[(ii.)]for given $\{\alpha_{i}\in \mathbb{R}\}$, $f(A)=\sum\limits_{1\leq i\leq n^{2}}\alpha_{i}A_{\downarrow i}$,
\item[(iii.)]for given $\{\alpha_{ij}\in \mathbb{R}\}$, $f(A)=\sum\limits_{j\in [n]}\sum\limits_{i\in [n]}\alpha_{ij}\boldsymbol{r}^{i}_{\downarrow j}$, where $\boldsymbol{r}^{i}$ is the $i^{\rm th}$ row of $A$,
\end{itemize}
then the following statements are equivalent: 
\begin{enumerate}
\item
For every $A\in V$, $\psi(A)\leq f(A)$.
\item
For every $A\in V_{\{0,1\}}$, $\psi(A)\leq f(A)$.
\end{enumerate}
\end{theorem}

A natural question is that which functions can be choose as $f$ in Theorems \ref{equiv0}, \ref{equimat2}, \ref{equilinear}, and \ref{equimat}. 
\begin{problem}
Under the same assumptions of Theorems \ref{equiv0} and \ref{equilinear}, what is a characterization of the functions $f:\mathbb{R}^{m}\rightarrow \mathbb{R}$ that the following statements are equivalent: 
\begin{enumerate}
	\item
	For every $\boldsymbol{a}\in U$,  $\varphi(\boldsymbol{a})\leq f(\boldsymbol{a})$.
	\item
	For every $\boldsymbol{a}\in U_{\{0,\pm 1\}}\, ( \boldsymbol{a}\in U_{\{0,1\}})$, $\varphi(\boldsymbol{a})\leq f(\boldsymbol{a})$.
\end{enumerate}
\end{problem}

\section{Weighted Grone-Merris Conjecture}
In this section using theorems of Section 1, we state an equivalent and weighted version of Grone-Merris Conjecture. 

\subsection{Laplacian of weighted graphs}

Suppose that $A=[\omega_{ij}]\in Sym_{n}(\mathbb{R})$ is a zero diagonal matrix. The matrix $A$ is the adjacency matrix of a (weighted) graph of order $n$ such that the weight of edge $i-j$ is $\omega_{ij}$. In this paper we look at every symmetric matrix with zero diagonal as the adjacency matrix of a weighted graph.

The Laplacian matrix of $A$, denoted by $L_{A}$, is $D-A$, where $D$ is a diagonal matrix with diagonal entries $d_{ii}=\sum_{j=1}^{n}\omega_{ij}$. It is well-known that zero is an eigenvalue of $L_{A}$ with corresponding eigenvector $\boldsymbol{j}_{n}$. For a symmetric matrix with nonnegative weights, the Laplacian matrix is positive semi-definite, but if some weights are negative, the Laplacian matrix is not necessary to be positive semi-definite.

\textbf{In this paper we arrange other Laplacian eigenvalues than zero (corresponding to the eigenvector $\boldsymbol{j}_{n}$) decreasingly and denote them by $\mu_{1}(L_{A}),\ldots,\mu_{n-1}(L_{A})$ such that $\mu_{1}(L_{A})\geq \ldots \geq \mu_{n-1}(L_{A})$.
}

More precisely, we have the following lemma:
\begin{lemma}\label{weightlem}
	Let  $A=[\omega_{ij}]$  be an $n\times n$ symmetric  matrix with zero diagonal and $\omega$ be a real number. If  $B=A+\omega (J_n-I_n)$, then for $i\in [n-1]$,
	\[
	\mu_{i}(L_{B})=\mu_{i}(L_{A})+n\omega.
	\]
\end{lemma}
\begin{proof}
	Suppose that $\mu_{i}(L_{A})$ is corresponding to the eigenvector  $\boldsymbol{\xi}$ perpendicular to $\boldsymbol{j}_{n}$.
	So, $L_{B}\boldsymbol{\xi}=L_{A}\boldsymbol{\xi}+\omega L_{(J_n-I_n)}\boldsymbol{\xi}=(\mu_{i}(L_{A})+n\omega)\boldsymbol{\xi}$.
\end{proof}

We have for $k\in [n-1]$, $$\sum_{i=1}^{k}\mu_{i}(L_{A})=\max_{\substack{U^{T}U=I_{k}\\U^{T}\boldsymbol{j}_{n}=\boldsymbol{0}} }tr(U^{T}L_{A}U)$$ 
and it is easy to see that $\psi(A)=\sum_{i=1}^{k}\mu_{i}(L_{A})$ is a convex function on  $Sym_{n}(\mathbb{R})$ (see \cite[Theporem 3.4]{cpw}).

For two vectors $\boldsymbol{x},\boldsymbol{y} \in \mathbb{R}^{n}$, we say that $\boldsymbol{y}$ majorizes  $\boldsymbol{x}$, written as $\boldsymbol{x} \prec \boldsymbol{y}$, if and only if
\[
\sum_{i=1}^{k} x_{\downarrow i} \leq \sum_{i=1}^{k} y_{\downarrow i} \quad  k\in [n-1],\qquad \sum_{i=1}^{n} x_{\downarrow i}=\sum_{i=1}^{n} y_{\downarrow i}.
\]

\begin{example}
	Let  $A=[\omega_{ij}]$  be an $n\times n$ symmetric  matrix with zero diagonal. If 
	$\omega_{1}\geq \omega_{2}\geq\cdots \geq\omega_{n^2 -n}$
	is the decreasing order of $(\omega_{ij})_{i,j,i\ne j}$ , then we show that
\[(\mu_{1}(L_{A}),\ldots,\mu_{n-1}(L_{A}))\prec (\displaystyle\sum_{i=1}^{n}\omega_{i},\displaystyle\sum_{i=n+1}^{2n}\omega_{i},\ldots,\displaystyle\sum_{i=n^{2}-2n+1}^{n^2 -n}\omega_{i}).\]
We have $\sum_{i=1}^{n-1}\mu_{i}(L_{A})=Tr(L_{A})=\sum_{i=1}^{n^2 -n}\omega_i$. So, to see other inequalities, it is sufficient to show that $\sum_{i=1}^{k}\mu_{i}(L_{A})\leq \sum_{j=1}^{k} \sum_{i=(j-1)n+1}^{jn} \omega_i=\sum_{i=1}^{kn} \omega_i$ for $k\in[n-1]$. By Lemma \ref{weightlem}, without loss of generality, we suppose that all weights are nonnegative. By Theorem \ref{equimat}, function (ii.), it is sufficient to show inequality for $(0,1)$ matrices. We know for simple graphs, if $2e(A)=qn+r$, $0\leq r\leq n-1,\,0\leq q$, then
\[(\mu_{1}(L_{A}),\ldots,\mu_{n-1}(L_{A}))\prec (n,\ldots,n,r,0,\ldots,0).\] 
Hence $\sum_{i=1}^{k}\mu_{i}(L_{A})\leq \sum_{j=1}^{k} \sum_{i=(j-1)n+1}^{jn} \omega_i=\sum_{i=1}^{kn} \omega_i$ is true and we have done.
\end{example}

\subsection{Young Tableau of Weights and Conjugate Degrees Sequence}
Suppose that  $A=[\omega_{ij}]$  is an $n\times n$ symmetric matrix with zero diagonal. Assume that for every $i\in [n]$,
	$\omega_{\downarrow i1}\geq \omega_{\downarrow i2}\geq\cdots \geq\omega_{\downarrow i(n-1)}$
is the decreasing order of  $(\omega_{ij})_{j\ne i}$ (the row $i$ of $A$ without diagonal entry). So, $d_{i}=\sum_{j=1}^{n}\omega_{ij}=\sum_{j=1}^{n-1}\omega_{\downarrow ij}$ is the degree of vertex $i$, for $i\in [n]$. We define the following Young tableau of these weights:
{\renewcommand{\arraystretch}{1.5}

\begin{table}[H]
$
\begin{blockarray}{c|cccc}
\text{SUM} &  \displaystyle\sum_{i=1}^{n}\omega_{\downarrow i(n-1)} & \cdots  & \displaystyle\sum_{i=1}^{n}\omega_{\downarrow i2}  & \displaystyle\sum_{i=1}^{n}\omega_{\downarrow i1} \\\cline{1-5}
	\begin{block}{c|c|c|c|c|}
	d_1 &\omega_{\downarrow 1(n-1)} &\cdots&\omega_{\downarrow 12}& \omega_{\downarrow 11} \\\hhline{~====}
	d_2  & \omega_{\downarrow 2(n-1)}&\cdots&\omega_{\downarrow 22}& \omega_{\downarrow 21}\\\hhline{~====}
		\vdots& \vdots  &\ddots  &\vdots&\vdots\\\hhline{~====}
		d_{n-1} & \omega_{\downarrow (n-1)(n-1)} &\cdots&\omega_{\downarrow (n-1)2}& \omega_{\downarrow (n-1)1}\\\hhline{~====}
		d_n &  \omega_{\downarrow n(n-1)} &\cdots&\omega_{\downarrow n2}& \omega_{\downarrow n1}  \\\cline{2-5}
	\end{block}
\end{blockarray}
$
\caption{Conjugate of the Weighted Degree Sequence}
\end{table}
}
In this tableau, define the conjugate of the degree sequence by $d_{j}^{*}= \displaystyle\sum_{i=1}^{n}\omega_{\downarrow ij}$, for $j\in [n-1]$.
This is a weighted version of the conjugate of the degree sequence. One can see easily that this version of $d_{j}^{*}$ for $(0,1)$ weights is $|\{i:\,d_i\geq j\}|$.
\begin{example}
	Suppose that
	\[
	A=\begin{pmatrix}
	0& -1 & 2.1 & 0 \\ 
	-1& 0 & -4 & 8 \\ 
	2.1& -4 & 0 & -7 \\ 
	0& 8 & -7  & 0
	\end{pmatrix}.
	\]
	Therefore, we have the following tableau:
\[
\begin{blockarray}{c|ccc}
\text{SUM} &  -19 &  -5  & 20.2 \\\cline{1-4}
\begin{block}{c|c|c|c|}
1.1 &-1 &0 &2.1\\\cline{2-4}
3  & -4&-1& 8\\\cline{2-4}
-8.9 & -7&-4& 2.1\\\cline{2-4}
1 &  -7&0&8  \\\cline{2-4}
\end{block}
\end{blockarray}\,.
\]	
	
\end{example}

By definitions above and theorems of Section 1, we state the following weighted version of Grone-Merris Conjecture. 
\begin{theorem}\label{ggm}
	Let  $A=[\omega_{ij}]$  be an $n\times n$ symmetric matrix with zero diagonal. Then
	\[
(\mu_{1}(L_{A}),\ldots,\mu_{n-1}(L_{A}))\prec (d_{1}^{*},\ldots,d_{n-1}^{*}).
	\]
\end{theorem}
\begin{proof}
Assume that $k\in [n-1]$. We must show $\sum_{j=1}^{k}\mu_{j}(L_{A})\leq \sum_{j=1}^{k}d_{j}^{*}= \sum_{j=1}^{k}\sum_{i=1}^{n}\omega_{\downarrow ij}$. 
From Lemma \ref{weightlem}, if we shift the entries of $A$ by $\omega$, two sides of inequality above change $kn\omega$. Thus without loss of generality, we suppose that $A$ is nonnegative. 
By Theorem \ref{equimat}, function (iii.), it is sufficient to show the inequality for $(0,1)$ matrices. For simple graphs Grone-Merris Conjecture is proved \cite{ba}. 
Hence $\sum_{i=1}^{k}\mu_{i}(L_{A})\leq \sum_{j=1}^{k}d_{j}^{*}$ is true and we have done.
\end{proof}

\begin{corollary}\label{ggmstar}
	Let $t,n\in\mathbb{N}$ and $S_{n}$  be a weighted star. If  
	$\omega_{1}\geq\cdots \geq  \omega_{t}>0\geq\omega_{t+1}\geq\cdots\geq \omega_{n-1}$
	is the decreasing order of the weights of edges, then \[(\mu_{1}(L_{S_{n}}),\ldots,\mu_{n-1}(L_{S_{n}}))\prec ((\sum_{i=1}^{t}\omega_{i})+\omega_{1},\omega_{2},\ldots,\omega_{n-2},(\sum_{i=t+1}^{n-1}\omega_{i})+\omega_{n-1})\]
\end{corollary}
\begin{figure}[H]
	\centering
	\definecolor{xdxdff}{rgb}{0.49019607843137253,0.49019607843137253,1.}
	\definecolor{ududff}{rgb}{0.30196078431372547,0.30196078431372547,1.}
	\begin{tikzpicture}[scale=.9,line cap=round,line join=round,>=triangle 45,x=1.0cm,y=1.0cm]
	\draw [line width=1.pt] (2.,3.)-- (0.,1.);
	\draw [line width=1.pt] (2.,3.)-- (1.,1.);
	\draw [line width=1.pt] (2.,3.)-- (3.,1.);
	\draw [line width=1.pt] (2.,3.)-- (4.,1.);
	\draw (1.6,1.66) node[anchor=north west] {$\cdots$};
	\begin{scriptsize}
	\draw [fill=ududff] (2.,3.) circle (1.5pt);
	\draw [fill=xdxdff] (0.,1.) circle (1.5pt);
	\draw [fill=ududff] (1.,1.) circle (1.5pt);
	\draw [fill=ududff] (3.,1.) circle (1.5pt);
	\draw [fill=ududff] (4.,1.) circle (1.5pt);
	\draw[color=black] (.45,1.7) node {$\omega_{1}$};
	\draw[color=black] (1.65,1.7) node {$\omega_{2}$};
	\draw[color=black] (2.51,1.7) node {$\omega_{n-2}$};
	\draw[color=black] (3.77,1.7) node {$\omega_{n-1}$};
	\end{scriptsize}
	\end{tikzpicture}
	\caption{A Weighted Star}
\end{figure}
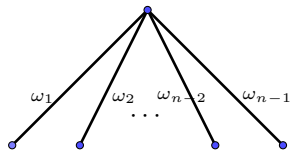
\begin{proof} 
	By Theorem \ref{ggm}, we have $d_{1}^{*}=(\sum_{i=1}^{t}\omega_{i})+\omega_{1}$, $d_{2}^{*}=\omega_{2}$,\ldots, $d_{n-1}^{*}=(\sum_{i=t+1}^{n-1}\omega_{i})+\omega_{n-1}$ and we have done.	
\end{proof}
\begin{remark}
	One can see that $\chi(\mu)=\Pi_{i=1}^{n-1}(\mu-\omega_i)\left(\mu-(\sum_{i=1}^{n-1}\omega_{i})-\sum_{i=1}^{n-1}\frac{\omega_{i}^{2}}{\mu-\omega_i}\right)$ is the characteristic polynomial of $L_{S_{n}}$. If $\mu_{1}\geq \cdots\geq \mu_{n}$ are the roots of $\chi$, then by interlacing (Theorem \ref{int}) we have $\mu_{1}\geq \omega_1\geq \mu_2 \geq \omega_2\geq \cdots\geq\omega_{n-1}\geq \mu_{n}$. By Corollary \ref{ggmstar} we obtain some useful inequalities for the roots of $\chi$. 
\end{remark}

\section{Weighted Brouwer's Conjecture}
For this section we need the definition of \textit{unordered majorization} \cite[14.E.6]{moa}. For two vectors $\boldsymbol{x},\boldsymbol{y} \in \mathbb{R}^{n}$, we say that $\boldsymbol{y}$ unordered majorizes  $\boldsymbol{x}$, written as $\boldsymbol{x} \trianglelefteq \boldsymbol{y}$, if and only if
\[
\sum_{i=1}^{k} x_{i} \leq \sum_{i=1}^{k} y_{i} \quad  k\in [n-1],\qquad \sum_{i=1}^{n} x_{i}=\sum_{i=1}^{n} y_{i}.
\]

The following conjecture states a weighted version of Brouwer's Conjecture.
\begin{conjecture}\label{gbrouwer}
	Let  $A=[\omega_{ij}]$  be an $n\times n$ symmetric nonnegative matrix with zero diagonal. If 
	$\omega_{1}\geq \omega_{2}\geq\cdots \geq\omega_{\binom{n}{2}}\geq 0$
	is the decreasing order of $(\omega_{ij})_{i<j}$, then 
	\[
	(\mu_{1}(L_{A}),\ldots,\mu_{n-1}(L_{A}))\trianglelefteq (e(A)+\omega_{1},\sum_{i=\binom{2}{2}+1}^{\binom{3}{2}}\omega_{i},\ldots,\sum_{i=\binom{k}{2}+1}^{\binom{k+1}{2}}\omega_{i},\ldots,\sum_{i=\binom{n-1}{2}+1}^{\binom{n}{2}}\omega_{i}).
	\]
\end{conjecture}

If $A$ is a $(0,1)$ matrix, Conjecture \ref{gbrouwer} is Brouwer's Conjecture.

Using Theorem \ref{equimat}, function (ii.), we can state weighted versions of the cases of Brouwer's Conjecture that is proven.
\begin{theorem}\cite{hmt}\label{btree}
	Let  $T$  be a tree of order $n$. For $k\in[n]$, 
	\[
	\sum_{i=1}^{k}\mu_{i}(L_{T})\leq e(T)+2k-1.
	\]
\end{theorem}
 The following theorem is a weighted version for trees.
\begin{theorem}\label{gbtree}
	Let  $T$  be a weighted tree. If 
	$\omega_{1}\geq \omega_{2}\geq\cdots \geq\omega_{n-1}\geq 0$ is the decreasing order of the weights of edges and $\omega_{n}=\cdots=\omega_{2n-1}=0$, then 
	\[
	(\mu_{1}(L_{T}),\ldots,\mu_{n}(L_{T}))\trianglelefteq (e(T)+\omega_{1},\omega_{2}+\omega_{3},\ldots,\omega_{2k-2}+\omega_{2k-1},\ldots,\omega_{2n-2}+\omega_{2n-1}).
	\]
\end{theorem}
\begin{proof}
Suppose that $k\in [n]$ and set $V=\{\sum_{ij\in E(T)}\omega_{ij}E_{ij}:\omega_{ij}\geq 0\}$. The function $\sum_{i=1}^{k}\mu_{i}(L_{T})$ is a convex function on $V$. If $2k-1\geq n-1$, then it is obvious that $\sum_{i=1}^{k}\mu_{i}(L_{T})\leq 2e(T)\leq e(T)+ \sum_{i=1}^{2k-1}\omega_i$. If $2k-1 < n-1$, then by Theorem \ref{equimat}, function (ii.), we have 
$\sum_{i=1}^{k}\mu_{i}(L_{T})\leq e(T)+ \sum_{i=1}^{2k-1}\omega_i=2(\sum_{i=1}^{2k-1}\omega_i)+\sum_{i=2k}^{n-1}\omega_i$.
\end{proof}

\subsection{Weight Tableau}
Assume that $A=[\omega_{ij}]$ is an $n\times n$ symmetric matrix with zero diagonal. If 
$\omega_{1}\geq \omega_{2}\geq\cdots \geq\omega_{\binom{n}{2}}$
is the decreasing order of $(\omega_{ij})_{i< j}$, then we define weight tableau $W(A)$ as below:
{\renewcommand{\arraystretch}{1.9}
\begin{equation}\label{wed}
\newcommand*{\temone}{\BAmulticolumn{1}{c|}{\omega_{1}}}
\newcommand*{\temthr}{\BAmulticolumn{1}{c|}{\omega_{3}}}
\newcommand*{\temsix}{\BAmulticolumn{1}{c|}{\omega_{6}}}
\newcommand*{\temn}{\BAmulticolumn{1}{|c}{\omega_{\binom{n}{2}}}}
\newcommand*{\temnmin}{\BAmulticolumn{1}{|c}{\omega_{\binom{n}{2}-2}}}
W(A)=
\begin{blockarray}{ccccccc}
\begin{block}{[ccccccc]}
\temone &\omega_{1} &  \omega_{2}&  \omega_{3}&\cdots&\omega_{n-2}&\omega_{n-1}\\ \cline{2-2}
\omega_{2}  &\temthr &  \omega_{n}&  \omega_{n+1}&\cdots&\omega_{2n-4}&\omega_{2n-3} \\\cline{3-3}
\omega_{4}  &\omega_{5} & \temsix &\ddots & &\vdots &\vdots \\
\vdots  &\vdots& & &\ddots & \vdots&\vdots\\
\omega_{\binom{n-2}{2}+1}&\omega_{\binom{n-2}{2}+2}&\cdots &\cdots & &\temnmin&\omega_{\binom{n}{2}-1}\\\cline{6-6}
\omega_{\binom{n-1}{2}+1}&\omega_{\binom{n-1}{2}+2}&\cdots &\cdots & & \omega_{\binom{n}{2}}&\temn\\ 
\end{block}
\end{blockarray}
\end{equation}
}

It is easy to see that $\sum_{i=1}^{n-1}r_{i}(W(A))=2e(A)$, where $r_{i}(W(A))=\sum_{j=1}^{n}W(A)_{ij}$ is the sum of row $i$, for $i\in [n-1]$. 
\begin{lemma}\label{brouwermaj}
Let  $A=[\omega_{ij}]$  be an $n\times n$ symmetric nonnegative matrix with zero diagonal. If 
$\omega_{1}\geq\cdots \geq\omega_{\binom{n}{2}}$
is the decreasing order of $(\omega_{ij})_{i< j}$, then 
\[
(r_{1}(W(A)),\cdots,r_{n-1}(W(A))) \trianglelefteq (e(A)+\omega_{1},\sum_{i=\binom{2}{2}+1}^{\binom{3}{2}}\omega_{i},\ldots,\sum_{i=\binom{n-1}{2}+1}^{\binom{n}{2}}\omega_{i}).
\]
\end{lemma}
\begin{proof}
From the definition of the weight tableau \ref{wed}, we have $\sum_{i=1}^{k}r_{i}(W(A))\leq e(A)+\sum_{i=1}^{\binom{k+1}{2}}\omega_{i}$, for $k\in [n-1]$. 
\end{proof}
Now, we state the following conjecture that allows us to extend Brouwer's Conjecture to real weights:

\begin{conjecture}[Weighted Brouwer's Conjecture]\label{strongmaj}
Let  $A=[\omega_{ij}]$  be an $n\times n$ symmetric matrix with zero diagonal. If 
$\omega_{1}\geq \omega_{2}\geq\cdots \geq\omega_{\binom{n}{2}}$
is the decreasing order of $(\omega_{ij})_{i<j}$, then 
\[
(\mu_{1}(L_{A}),\ldots,\mu_{n-1}(L_{A}))\trianglelefteq (r_{1}(W(A)),\ldots,r_{n-1}(W(A))).
\]

\end{conjecture}
\begin{remark}
Brouwer's Conjecture for simple graphs, Conjecture \ref{brouwer}, and Conjecture \ref{strongmaj} are equivalent: For $k\in [n-1]$, Conjecture \ref{strongmaj} states $\sum_{i=1}^{k}\mu_{i}(L_{A})\leq \displaystyle\sum_{i=1}^{kn-\binom{k+1}{2}}\omega_{i}+\displaystyle\sum_{i=1}^{\binom{k+1}{2}}\omega_{i}$.
From Lemma \ref{weightlem}, if we shift the entries of $A$ by $\min_{i}\omega_i$, two sides of inequality above change $kn\omega$. Thus without loss of generality, we suppose that $A$ is nonnegative. 
By Theorem \ref{equimat}, function (ii.), it is sufficient to show the inequality for $(0,1)$ matrices. For simple graphs, Brouwer's Conjecture is obvious for $e(A)\geq kn-\binom{k+1}{2}$ and for $e(A)< kn-\binom{k+1}{2}$  we have $\sum_{i=1}^{k}r_{i}(W(A))=\displaystyle\sum_{i=1}^{kn-\binom{k+1}{2}}\omega_{i}+\displaystyle\sum_{i=1}^{\binom{k+1}{2}}\omega_{i}=e(A)+\displaystyle\sum_{i=1}^{\binom{k+1}{2}}\omega_{i}$. 
Hence, two conjectures are equivalent.

\end{remark}

\subsection{Optimum Arrangements of Weights of Graphs } A question about these bounds on sum of the eigenvalues of graphs and symmetric matrices is about the matrices and graphs that achieve these bounds.
In other way, we want to know for which (weighted) graphs these bounds are sharp. We state these questions as the following problems.
\begin{problem}[Laplacian Matrix]
	Let $k, n\in\mathbb{N}$, $k<n$, and $\omega_{1}, \omega_{2},\cdots ,\omega_{\binom{n}{2}}$ be given nonnegative real numbers. What is the graph of order $n$ with this weights that have the maximum (or minimum) of $S_{k}(L_{A})=\mu_1(L_{A})+\cdots+\mu_k(L_{A})$?
	
\end{problem}
It seems that for a simple graph $G$, $k<n$, and $\binom{k+1}{2}<e(G)<kn-\binom{k+1}{2}$ the extremal graphs for problem above are  threshold graphs with clique number $k+1$ which achieve the upper bound in Brouwer's Conjecture. 
\begin{problem}[Adjacency Matrix]
	Let $k, n\in\mathbb{N}$, $k<n$, and $\omega_{1}, \omega_{2},\cdots ,\omega_{\binom{n}{2}}$ be given nonnegative real numbers. What is the graph of order $n$ with this weights that have the maximum (or minimum) of $S_{k}(A)=\lambda_1(A)+\cdots+\lambda_k(A)$?
	
\end{problem}

Extremal simple graphs with given $e$ edges for $k=1$, spectral radius, have been considered in recent decades (see \cite{bs},\cite{oad}, and \cite{ro}).

\begin{problem}[Nonsymmetric Matrices: Singular Values]
	Let $k,m, n\in\mathbb{N}$, $k<m,n$, and $\omega_{1}, \omega_{2},\cdots ,\omega_{mn}$ be given real numbers. What is the arrangement of these weights in an $n\times m$ matrix $M$ such that have the maximum (or minimum) sum of the singular values $S_{k}(M)=\sigma_1(M)+\cdots+\sigma_k(M)$?
\end{problem}

\section{Positive semi-definite decomposition}

In this section, we consider a relationship between the spectral decomposition and  any positive semi-definite decomposition of a symmetric matrix.

\begin{lemma}\label{decpos}
Let $t,n\in\mathbb{N}$ and $A\in Sym_{n}(\mathbb{R})$. If $A=\sum_{i=1}^{t}\theta_{i} \boldsymbol{u}_{i}\boldsymbol{u}_{i}^{T}$ for some real vectors $\boldsymbol{u}_{i},\,i\in[t]$.  If $\theta_{i}> 0 \text{ for } i\in[t]$  and $r=\text{rank}(span\{\boldsymbol{u}_{1},\ldots,\boldsymbol{u}_{t}\})$, then $A$ has  exactly $r$ nonzero eigenvalues and $span\{\boldsymbol{u}_{1},\ldots,\boldsymbol{u}_{t}\}=span\{\boldsymbol{\xi}_{1},\ldots,\boldsymbol{\xi}_{r}\}$ for eigenvectors $\boldsymbol{\xi}_{i},\,i\in[r]$, corresponding to  nonzero eigenvalues of $A$.  
\end{lemma}
\begin{proof}
Suppose that ${\boldsymbol{\xi}}$ is a $0$-eigenvector of $A$. So, $\boldsymbol{\xi}^{T}A\boldsymbol{\xi}=0$ and $\boldsymbol{\xi}\bullet\boldsymbol{u}_{i}=0$. Therefore 
$span\{\boldsymbol{u}_{1},\ldots,\boldsymbol{u}_{t}\}\subseteq Im(A)$. Also, if $\boldsymbol{v}\in (span\{\boldsymbol{u}_{1},\ldots,\boldsymbol{u}_{t}\})^{\bot}$, then $A\boldsymbol{v}=\boldsymbol{0}$. Thus $span\{\boldsymbol{u}_{1},\ldots,\boldsymbol{u}_{t}\}\supseteq Im(A)$ and hence $span\{\boldsymbol{u}_{1},\ldots,\boldsymbol{u}_{t}\}=span\{\boldsymbol{\xi}_{1},\ldots,\boldsymbol{\xi}_{r}\}$ for eigenvectors $\boldsymbol{\xi}_{i},\,i\in[r]$ corresponding to nonzero eigenvalues of $A$.  
\end{proof}

\begin{theorem}\label{decmaj}
Let $t,n\in\mathbb{N}$ and $A\in Sym_{n}(\mathbb{R})$. Suppose that $A=\sum_{i=1}^{t}\theta_{i} \boldsymbol{u}_{i}\boldsymbol{u}_{i}^{T}$ for some unit real vectors $\boldsymbol{u}_{i},\,i\in[t]$.  If $\theta_{i}> 0,\,i\in[t]$, then for $m=min(n,t)$
\[
(\theta_1,\ldots,\theta_{m})\prec (\lambda_1(A),\ldots,\lambda_{m}(A)).
\]  
\end{theorem}
\begin{proof}
For every $k\in [m]$, set $A_k=\sum_{i=1}^{k}\theta_{i} \boldsymbol{u}_{i}\boldsymbol{u}_{i}^{T}$. By Lemma \ref{decpos}, we have $rank(A_k)\leq k$ and hence,
\[
\sum_{i=1}^{k}\theta_i=Tr(A_k)=\sum_{i=1}^{rank(A_k)}\lambda_i(A_k)\leq \sum_{i=1}^{rank(A_k)}\lambda_i(A)\leq \sum_{i=1}^{k}\lambda_i(A)
\]
and we have done.
\end{proof}

By Theorem \ref{decmaj} we state the following conjecture that is similar to Brouwer's Conjecture.
\begin{conjecture}\label{biii}
Let $n,t\in\mathbb{N}$ and $\theta_{1},\ldots,\theta_{t}$ be nonnegative numbers. Suppose that  $\boldsymbol{u}_{1},\ldots,\boldsymbol{u}_{t}$ are $t$ unit vectors of $\boldsymbol{j}_{n}^{\perp}$ and $L=\sum_{i=1}^{t}\theta_{i} \boldsymbol{u}_{i}\boldsymbol{u}_{i}^{T}$. If $A=[\omega_{ij}]=-L+Diag(L_{11},\ldots,L_{nn})$, then for $k\in min(t,n-1)$ we have 
\[
\sum_{i=1}^{k}\theta_{i}\leq \sum_{i=1}^{k}r_{i}(W(A)).
\]
\end{conjecture}
By Theorem \ref{decmaj} we state the following theorem that is  similar to Grone-Merris Conjecture.
\begin{theorem}\label{gmiii}
	Let $n,t\in\mathbb{N}$ and $\theta_{1},\ldots,\theta_{t}$ be nonnegative numbers. Suppose that  $\boldsymbol{u}_{1},\ldots,\boldsymbol{u}_{t}$ are $t$ unit vectors of $\boldsymbol{j}_{n}^{\perp}$ and $L=\sum_{i=1}^{t}\theta_{i} \boldsymbol{u}_{i}\boldsymbol{u}_{i}^{T}$. If $A=[\omega_{ij}]=-L+Diag(L_{11},\ldots,L_{nn})$, then we for $k\in min(t,n-1)$ have 
		\item%
		\[
		\sum_{i=1}^{k}\theta_{i}\leq \sum_{j=1}^{k}\sum_{i=1}^{n}\omega_{\downarrow ij}.
		\]
\end{theorem}

Similar to Lemma \ref{decpos}, we state the following proposition.
\begin{proposition}
	Let $t,n\in\mathbb{N}$ and $A\in Sym_{n}(\mathbb{R})$. Suppose that $A=\sum_{i=1}^{t}\theta_{i} \boldsymbol{u}_{i}\boldsymbol{u}_{i}^{T}$ for some independent real vectors $\boldsymbol{u}_{i},\,i\in[t]$.  If $\theta_{i}\neq 0,\,i\in[t]$, then $A$ has  exactly $t$ nonzero eigenvalues and $span\{\boldsymbol{u}_{1},\ldots,\boldsymbol{u}_{t}\}=span\{\boldsymbol{\xi}_{1},\ldots,\boldsymbol{\xi}_{t}\}$ for eigenvectors $\boldsymbol{\xi}_{i},\,i\in[t]$ corresponding to nonzero eigenvalues of $A$.  
\end{proposition}
\begin{proof}
	For every $i\in [t]$, suppose that $V_i=span\{\boldsymbol{u}_{1},\ldots,\boldsymbol{u}_{i-1},\boldsymbol{u}_{i+1},\ldots,\boldsymbol{u}_{t}\}$ and $\widehat{\boldsymbol{u}_{i}}=\boldsymbol{u}_{i}-\boldsymbol{u'}_{i}$, where $\boldsymbol{u'}_{i}\in V_i$ and $\widehat{\boldsymbol{u}_{i}}\in {V_{i}}^{\perp}$. So, $A\widehat{\boldsymbol{u}_{i}}=(\theta_{i}|\widehat{\boldsymbol{u}_{i}}|^{2})\boldsymbol{u}_{i}$. Therefore
	$span\{\boldsymbol{u}_{1},\ldots,\boldsymbol{u}_{t}\}\subseteq Im(A)$. Also, if $\boldsymbol{v}\in (span\{\boldsymbol{u}_{1},\ldots,\boldsymbol{u}_{t}\})^{\bot}$, then $A\boldsymbol{v}=\boldsymbol{0}$. Thus $span\{\boldsymbol{u}_{1},\ldots,\boldsymbol{u}_{t}\}=span\{\boldsymbol{\xi}_{1},\ldots,\boldsymbol{\xi}_{t}\}$ for eigenvectors $\boldsymbol{\xi}_{i},\,i\in[t]$ corresponding to nonzero eigenvalues of $A$.  
\end{proof}
The next proposition considers a relation between the sign of entries of a positive semi-definite matrix and its positive semi-definite decomposition.
\begin{proposition}
	Let $n,t\in\mathbb{N}$ and $\theta_{1},\ldots,\theta_{t}$ be positive real numbers. Suppose that  $\boldsymbol{u}_{1},\ldots,\boldsymbol{u}_{t}$ are $t$ real vectors of $\boldsymbol{j}_{n}^{\perp}$ and $L=\sum_{i=1}^{t}\theta_{i} \boldsymbol{u}_{i}\boldsymbol{u}_{i}^{T}$. If $t<n-1$ and  $A=[\omega_{ij}]=-L+Diag(L_{11},\ldots,L_{nn})$ is irreducible, then some of the entries of $A$ are negative.
\end{proposition}
\begin{proof}
	By Lemma \ref{decpos}, the number of positive eigenvalues of $L$ are at most $t$. Suppose by contradiction that all entries of $A$ are nonnegative. Matrix $A$ is irreducible, hence, it is the adjacency matrix of a connected weighted graph. Suppose that $G$ is the simple graph such that $ij\in E(G)$ if and only if $\omega_{ij}\neq 0$. We have
	\[
	\mu_{n-1}(L_A)\geq (\min_{\omega_{ij}>0}\omega_{ij})\mu_{n-1}(L_G)>0,
	\]
	which is a contradiction and  the proof is complete.
\end{proof}

\section{Some Conjectured Bounds on Sum of the Largest Eigenvalues of Graphs}
Let $n\in\mathbb{N}$ and $\phi$ be a function on graphs. For given edges of simple graphs on $n$ vertices, an extremal problem is that which graph has the maximum (or minimum) value of  $\phi$. In the other hand, the problem is that what distribution and arrangement of the given weights $0$ and $1$  in an adjacency matrix has the extreme value.  
In this section, we state some conjectures on the sum of the $k$ largest(smallest) eigenvalues of simple graphs. In particular, we conjecure some upper and lower bounds for trees and other graphs.

For conjectured bounds in this section, we use the data of trees and graphs from \cite{graphdata} and \cite{treedata}. 
\subsection{Bounds for Trees}
\subsubsection{Adjacency Matrix}
It is well-known that for the largest adjacency eigenvalue of trees we have:
\begin{theorem}\cite[p. 21]{cds}
Let $n\in \mathbb{N}$ and $T$ be a tree on $n$ vertices, we have
\[
\lambda_1(P_{n})\leq \lambda_1(T)\leq \lambda_1(S_{n})
\]	
\end{theorem}	
The \textit{subdivision} of an edge $e=ij$ of a simple graph $G$ is a graph obtained by deleting $e$ and adding a new vertex $u$ and two new edges $iu$ and $uj$ . The subdivision of $G$ is a graph by the subdivision of all edges and is denoted by $S(G)$.
Suppose that $k\in \mathbb{N}$ and $T$ is a tree of order $k$ and $V(T)=\{v_1,\ldots,v_k\}$. Now, we add some new vertices to $v_i$ in $S(T)$, for all $i\in [k]$, such that $|d_{v_i}-d_{v_j}|\leq 1$ for all $i,j\in[k]$ and all new vertices have degree one. We denote the set of all such trees of order $n$ by  $\mathcal{T}_{n,k}$.
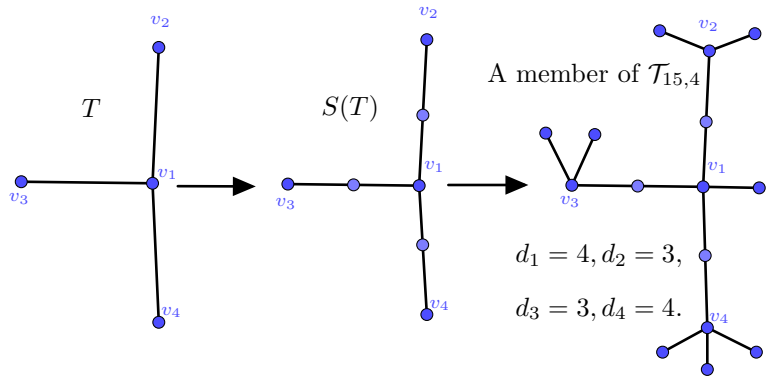
\begin{figure}[H]
\centering
\definecolor{xdxdff}{rgb}{0.49019607843137253,0.49019607843137253,1.}
\definecolor{ududff}{rgb}{0.30196078431372547,0.30196078431372547,1.}
\begin{tikzpicture}[scale=1.5,line cap=round,line join=round,>=triangle 45,x=1.0cm,y=1.0cm]

\draw [->,line width=1.pt] (0.5468835018448248,6.493112810200549) -- (1.2383188433163177,6.490810881325153);
\draw [line width=1.pt] (0.3229885461003082,6.518677359908862)-- (0.3783908368802463,7.7236771843725185);
\draw [line width=1.pt] (-0.8404595602783911,6.532527932603847)-- (0.3229885461003082,6.518677359908862);
\draw [line width=1.pt] (0.3229885461003082,6.518677359908862)-- (0.3783908368802463,5.285976390055237);
\draw [line width=1.pt] (2.6859842251178363,6.497868969173132)-- (2.754168790751931,7.792344717615887);
\draw [line width=1.pt] (1.5225361187391369,6.511719541868117)-- (2.6859842251178363,6.497868969173132);
\draw [line width=1.pt] (2.6859842251178363,6.497868969173132)-- (2.754168790751931,5.354643923298605);
\draw [->,line width=1.pt] (2.9447525363462352,6.50359315617931) -- (3.6541687907519322,6.50359315617931);
\draw [line width=1.pt] (5.202984549771504,6.487344893152231)-- (5.258386840551442,7.692344717615887);
\draw [line width=1.pt] (4.039536443392804,6.501195465847215)-- (5.202984549771504,6.487344893152231);
\draw [line width=1.pt] (5.202984549771504,6.487344893152231)-- (5.2399377908127285,5.238147945617792);
\draw [line width=1.pt] (4.818122720625557,7.871296565574082)-- (5.258386840551442,7.692344717615887);
\draw [line width=1.pt] (5.258386840551442,7.692344717615887)-- (5.661752860999899,7.845732015865768);
\draw [line width=1.pt] (5.2399377908127285,5.238147945617792)-- (5.674535135854056,5.020849273097126);
\draw [line width=1.pt] (5.2399377908127285,5.238147945617792)-- (4.843687270333871,5.020849273097126);
\draw [line width=1.pt] (4.039536443392804,6.501195465847215)-- (3.8083230071471785,6.9637550509289525);
\draw (-0.38503006316158195,7.347223296553655) node[anchor=north west] {$T$};
\draw (1.7368275626284286,7.385570121116125) node[anchor=north west] {$S(T)$};
\draw (3.2092176478956,7.666780167907573) node[anchor=north west] {$\text{A member of } \mathcal{T}_{15,4}$};
\draw (3.3386515897776,6.235165384242018) node[anchor=north west] {$\begin{array}{c} d_1=4, d_2=3,\\ d_3=3, d_4=4.\\ \end{array}$};
\draw [line width=1.pt] (5.202984549771504,6.487344893152231)-- (5.700099685562369,6.478028606470996);
\draw [line width=1.pt] (4.039536443392804,6.501195465847215)-- (4.242920352188507,6.950972776074796);
\draw [line width=1.pt] (5.2399377908127285,5.238147945617792)-- (5.2399377908127285,4.867461974847246);
\begin{scriptsize}
\draw [fill=ududff] (0.3783908368802463,7.7236771843725185) circle (1.5pt);
\draw[color=ududff] (0.3946887029419761,7.9543813521261) node {$v_2$};
\draw [fill=ududff] (0.3229885461003082,6.518677359908862) circle (1.5pt);
\draw[color=ududff] (0.45747097779613284,6.601718416418739) node {$v_1$};
\draw [fill=ududff] (-0.8404595602783911,6.532527932603847) circle (1.5pt);
\draw[color=ududff] (-0.8579742327653794,6.38216154506482) node {$v_3$};
\draw [fill=ududff] (0.3783908368802463,5.285976390055237) circle (1.5pt);
\draw[color=ududff] (0.49138235206691625,5.3746200304196915) node {$v_4$};
\draw [fill=ududff] (2.754168790751931,7.792344717615887) circle (1.5pt);
\draw[color=ududff] (2.7721918258151206,8.031075001251041) node {$v_2$};
\draw [fill=ududff] (2.6859842251178363,6.497868969173132) circle (1.5pt);
\draw[color=ududff] (2.821918258151206,6.668936141564582) node {$v_1$};
\draw [fill=ududff] (1.5225361187391369,6.511719541868117) circle (1.5pt);
\draw[color=ududff] (1.5067466152536082,6.338216154506482) node {$v_3$};
\draw [fill=ududff] (2.754168790751931,5.354643923298605) circle (1.5pt);
\draw[color=ududff] (2.8688854749400607,5.438531404690475) node {$v_4$};
\draw [fill=xdxdff] (2.718927551245164,7.123294039323354) circle (1.5pt);
\draw [fill=xdxdff] (2.7173011344558975,5.972790092719148) circle (1.5pt);
\draw [fill=xdxdff] (2.1046458143578715,6.504789664539323) circle (1.5pt);
\draw [fill=ududff] (5.258386840551442,7.692344717615887) circle (1.5pt);
\draw[color=ududff] (5.251953147521519,7.91603452756363) node {$v_2$};
\draw [fill=ududff] (5.202984549771504,6.487344893152231) circle (1.5pt);
\draw[color=ududff] (5.3299972083989,6.656153866710425) node {$v_1$};
\draw [fill=ududff] (4.039536443392804,6.501195465847215) circle (1.5pt);
\draw[color=ududff] (4.024854761522477,6.356596995356507) node {$v_3$};
\draw [fill=ududff] (5.2399377908127285,5.238147945617792) circle (1.5pt);
\draw[color=ududff] (5.348646796646458,5.323490931003064) node {$v_4$};
\draw [fill=xdxdff] (5.229492102422493,7.063884163311259) circle (1.5pt);
\draw [fill=xdxdff] (5.221023628279175,5.877537439917533) circle (1.5pt);
\draw [fill=xdxdff] (4.622902715581441,6.494250629273541) circle (1.5pt);
\draw [fill=ududff] (4.818122720625557,7.871296565574082) circle (1.5pt);
\draw [fill=ududff] (5.661752860999899,7.845732015865768) circle (1.5pt);
\draw [fill=ududff] (5.674535135854056,5.020849273097126) circle (1.5pt);
\draw [fill=ududff] (4.843687270333871,5.020849273097126) circle (1.5pt);
\draw [fill=ududff] (3.8083230071471785,6.9637550509289525) circle (1.5pt);
\draw [fill=ududff] (5.700099685562369,6.478028606470996) circle (1.5pt);
\draw [fill=ududff] (4.242920352188507,6.950972776074796) circle (1.5pt);
\draw [fill=ududff] (5.2399377908127285,4.867461974847246) circle (1.5pt);
\end{scriptsize}
\end{tikzpicture}
\caption{A member of $\mathcal{T}_{15,4}$.}
\end{figure}

\begin{example}
Suppose that $k=2$ and $n\geq 3$. The only tree of order $2$ is $P_2$ and $S(P_2)=P_3$. Hence, 
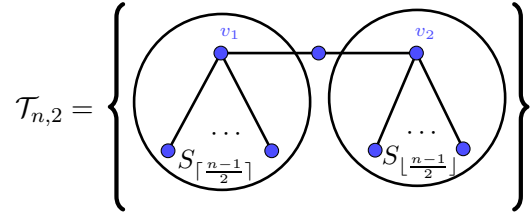
\begin{center}
	\definecolor{ududff}{rgb}{0.30196078431372547,0.30196078431372547,1.}
\begin{tikzpicture}[scale=1.3,line cap=round,line join=round,>=triangle 45,x=1.0cm,y=1.0cm]
\draw [line width=1.pt] (2.012760572694985,1.5) circle (0.9cm);
\draw [line width=1.pt] (4.007240572694984,1.5415517180849538) circle (0.9cm);
\draw [line width=1.pt] (2.,2.)-- (3.,2.);
\draw [line width=1.pt] (3.,2.)-- (4.,2.);
\draw [line width=1.pt] (2.,2.)-- (1.46,1.);
\draw [line width=1.pt] (2.,2.)-- (2.52,1.);
\draw [line width=1.pt] (4.,2.)-- (3.577873129421675,1.0061494273050153);
\draw [line width=1.pt] (4.,2.)-- (4.47816035459567,1.02);
\draw (3.8,1.3385631719846447) node[anchor=north west] {$\cdots$};
\draw (1.8,1.3247125992896602) node[anchor=north west] {$\cdots$};
\draw (1.46,1.1) node[anchor=north west] {$S_{\lceil\frac{n-1}{2}\rceil}$};
\draw (3.55,1.15) node[anchor=north west] {$S_{\lfloor\frac{n-1}{2}\rfloor}$};
\begin{scriptsize}
\draw [fill=ududff] (2.,2.) circle (2.pt);
\draw[color=ududff] (2.095861851058329,2.2) node {$v_1$};
\draw [fill=ududff] (3.,2.) circle (2.pt);
\draw [fill=ududff] (4.,2.) circle (2.pt);
\draw[color=ududff] (4.090344319136102,2.2) node {$v_2$};
\draw[color=black] (0.3,1.4) node{{\large $\mathcal{T}_{n,2}=$}};
\draw [fill=ududff] (1.46,1.) circle (2.pt);
\draw [fill=ududff] (2.52,1.) circle (2.pt);
\draw [fill=ududff] (3.577873129421675,1.0061494273050153) circle (2.pt);
\draw [fill=ududff] (4.47816035459567,1.02) circle (2.pt);
\end{scriptsize}
\draw [decoration={brace,amplitude=0.5em},decorate,ultra thick,black](1.,0.4) -- (1.,2.5);
\draw [decoration={brace,amplitude=0.5em},decorate,ultra thick,black](5.,2.5) -- (5.,0.4);
\end{tikzpicture}  
\end{center}
and $|\mathcal{T}_{n,2}|=1$.
\end{example}

We conjecture that for an integer $k$, we have the following statement:
\begin{conjecture}\label{adjcon1}
	Let $k\in \mathbb{N},k\geq 2$. 
\begin{itemize}
	\item For every $n\geq k+2$ and every tree $T$ on $n$ vertices: $\sum_{i=1}^{k}\lambda_i(S_{n})\leq \sum_{i=1}^{k}\lambda_i(T)$;
	\item 
	for every $n\geq 3k-1$ and every tree $T$ on $n$ vertices, we have
	\[
	\sum_{i=1}^{k}\lambda_i(T)\leq \sum_{i=1}^{k}\lambda_i(T_{n,k}),
	\]
	where $T_{n,k}\in \mathcal{T}_{n,k}$.
\end{itemize}

\end{conjecture}

Now, we consider the set $\mathcal{T}_{n,k}$ for $k>1$ and $n\geq k^2+1$. Assume that $T_{n,k}$ is a member of $\mathcal{T}_{n,k}$. By the definition of $\mathcal{T}_{n,k}$, the vertices $v_1,\ldots,v_k$ of $T_{n,k}$ have degree at least $k$. The matching number of $T_{n,k}$ is $k$ and hence $T_{n,k}$ has exactly $k$ positive eigenvalues. We define for $T_{n,k}$ the function
$f_{T_{n,k}}(x)=\sum_{i=1}^{k}\sqrt{x+\lambda_{i}^{2}(T_{n,k})}$. By using the Taylor series of $f_{T_{n,k}}(x)$, one can see that for two trees $T$ and $T'$, there exists $X_0\in\mathbb{R}$ such that for every $x>X_0$, $f_{T}(x)\leq f_{T'}(x)$ or  for every $x>X_0$, $f_{T'}(x)\leq f_{T}(x)$. For an integer $s$, where $0\leq s\leq k-1$, define $T^{s}_{k}$ as follows: $T^{s}_{k}\in \mathcal{T}_{k^2 +1+s,k}$ and there exists $X_0\in\mathbb{R}$ such that for every $x>X_0$ and $T\in \mathcal{T}_{k^2 +1+s,k}$, $f_{T}(x)\leq f_{T^{s}_{k}}(x)$.

The following lemma is a corollary of \cite[Lemma 2.8]{cs} and we use it to construct our extremal trees.

\begin{lemma}
Let $k,r\in \mathbb{N}$ and $G$ be a bipartite graph with partitions $X,Y$ and $\lambda_1(G)\geq\cdots\geq \lambda_k(G)>0$. If $G'$ is obtained from $G$ by joining each vertex  of $X$ to $r$ new pendant vertices, then
$\lambda_i(G')=\sqrt{r+\lambda_{i}^{2}(G)}$, for $i\in[k]$.
\end{lemma}
Now we can state a stronger version of Conjecture \ref{adjcon1}.

\begin{conjecture}\label{adjcon2}
	Let $k\in \mathbb{N},k\geq 2$. There exists an integer $M_k$ such that for every $n\geq M_k$ and every tree $T$ on $n$ vertices, we have
		\[
		\sum_{i=1}^{k}\lambda_i(T)\leq \sum_{i=1}^{k}\lambda_i(T^{s}_{n,k}),
		\]
		where $0\leq s\leq k-1$, $n-k^2-1\equiv s\text{ mod }k$, and $T^{s}_{n,k}$ is obtained from $T^{s}_{k}$ by joining each vertex  of $\{v_1,\ldots,v_k\}$($\subseteq V(T^{s}_{k})$) to $\frac{n-k^2-1-s}{k}$ new pendant vertices.
\end{conjecture}

\subsubsection{Laplacian Matrix: Sum of the Largest Eigenvalues}

For the largest Laplacian eigenvalue of trees we have the following well-known theorem(see \cite[Theorem 3.5]{zh} for a stronger result):
\begin{theorem}
	Let $n\in \mathbb{N}$ and $T$ be a tree on $n$ vertices, we have
	\[
	\mu_1(L_{P_{n}})\leq \mu_1(L_T)\leq \mu_1(L_{S_{n}})
	\]	
\end{theorem}
In \cite{gzw} it is shown that for every tree $T$ of order $n$, $\sum_{i=1}^{2}\mu_i(L_T)\leq \sum_{i=1}^{2}\mu_i(L_{S_{n,2}})$, where $S_{n,2}$ is the tree shown below:

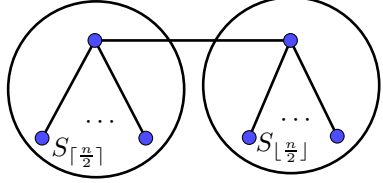
\begin{center}
	\definecolor{ududff}{rgb}{0.30196078431372547,0.30196078431372547,1.}
	\begin{tikzpicture}[scale=1.3,line cap=round,line join=round,>=triangle 45,x=1.0cm,y=1.0cm]
	\draw [line width=1.pt] (2.012760572694985,1.5) circle (0.9cm);
	\draw [line width=1.pt] (4.007240572694984,1.5415517180849538) circle (0.9cm);
	\draw [line width=1.pt] (4.,2.)-- (2.,2.);
	\draw [line width=1.pt] (2.,2.)-- (1.46,1.);
	\draw [line width=1.pt] (2.,2.)-- (2.52,1.);
	\draw [line width=1.pt] (4.,2.)-- (3.577873129421675,1.0061494273050153);
	\draw [line width=1.pt] (4.,2.)-- (4.47816035459567,1.02);
	\draw (3.8,1.3385631719846447) node[anchor=north west] {$\cdots$};
	\draw (1.8,1.3247125992896602) node[anchor=north west] {$\cdots$};
	\draw (1.46,1.1) node[anchor=north west] {$S_{\lceil\frac{n}{2}\rceil}$};
	\draw (3.55,1.15) node[anchor=north west] {$S_{\lfloor\frac{n}{2}\rfloor}$};
	\begin{scriptsize}
	\draw [fill=ududff] (2.,2.) circle (2.pt);
	\draw [fill=ududff] (4.,2.) circle (2.pt);
	\draw [fill=ududff] (1.46,1.) circle (2.pt);
	\draw [fill=ududff] (2.52,1.) circle (2.pt);
	\draw [fill=ududff] (3.577873129421675,1.0061494273050153) circle (2.pt);
	\draw [fill=ududff] (4.47816035459567,1.02) circle (2.pt);
	\end{scriptsize}
	\end{tikzpicture}  
\end{center}

Suppose that $k\in \mathbb{N}$ and $S_k$ is the star of order $k$ and $V(S_k)=\{v,v_1,\ldots,v_{k-1}\}$. Now, we add some new vertices to $v_i$, for all $i\in [k-1]$, such that $|d_{v_i}-d_{v_j}|\leq 1$ for all $i,j\in[k-1]$ and all new vertices have degree one. We denote the set of all such trees of degree $n$ by  $\mathcal{S}_{n,k}$.
	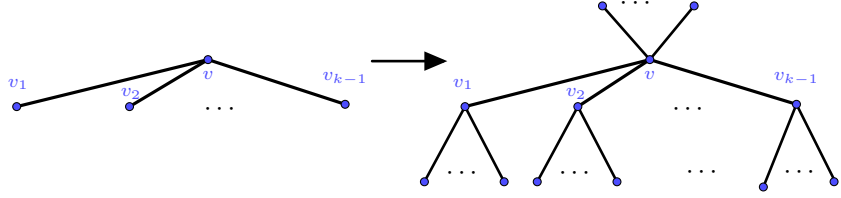
\begin{figure}[H]
	\centering
	\definecolor{ududff}{rgb}{0.30196078431372547,0.30196078431372547,1.}
	\begin{tikzpicture}[scale=1.,line cap=round,line join=round,>=triangle 45,x=1.0cm,y=1.0cm]
	\draw [line width=1.pt] (1.9800574450347974,2.4307471365250777)-- (1.440057445034797,1.4307471365250777);
	\draw [line width=1.pt] (1.9800574450347974,2.4307471365250777)-- (2.5000574450347997,1.4307471365250777);
	\draw [line width=1.pt] (3.4800574450347974,2.4307471365250777)-- (2.940057445034797,1.4307471365250777);
	\draw [line width=1.pt] (3.4800574450347974,2.4307471365250777)-- (4.0000574450347997,1.4307471365250777);
	\draw [line width=1.pt] (4.436608636510717,3.0560341861627247)-- (3.8072408829661954,3.770746434379706);
	\draw [line width=1.pt] (4.436608636510717,3.0560341861627247)-- (5.026091280124835,3.770746434379706);
	\draw [line width=1.pt] (6.386149427305012,2.461494273050155)-- (5.946149427305012,1.3614942730501547);
	\draw [line width=1.pt] (6.386149427305012,2.461494273050155)-- (6.886149427305012,1.4614942730501548);
	\draw (6.091943652220525,1.7594826436141127) node[anchor=north west] {$\cdots$};
	\draw (3.9149994733909635,4.011723702116658) node[anchor=north west] {$\cdots$};
	\draw (4.795056918425762,1.7594251985793136) node[anchor=north west] {$\cdots$};
	\draw (1.6063792074442107,1.7540803528341744) node[anchor=north west] {$\cdots$};
	\draw (3.1063792074442107,1.7540803528341744) node[anchor=north west] {$\cdots$};

	\draw [line width=1.2pt] (4.436608636510717,3.0560341861627247)-- (1.9800574450347974,2.4307471365250777);
	\draw [line width=1.2pt] (4.436608636510717,3.0560341861627247)-- (3.4800574450347974,2.4307471365250777);
	\draw [line width=1.2pt] (4.436608636510717,3.0560341861627247)-- (6.386149427305012,2.461494273050155);
	\draw (4.61206345730778,2.60459560278387) node[anchor=north west] {$\cdots$};
	\draw [->,line width=1.pt] (0.7447525363462352,3.0359315617931) -- (1.7447525363462352,3.0359315617931);
	
		\draw (-1.61206345730778,2.60459560278387) node[anchor=north west] {$\cdots$};
	
	\draw [line width=1.2pt] (-1.436608636510717,3.0560341861627247)-- (-3.9800574450347974,2.4307471365250777);
	\draw [line width=1.2pt] (-1.436608636510717,3.0560341861627247)-- (-2.4800574450347974,2.4307471365250777);
	\draw [line width=1.2pt] (-1.436608636510717,3.0560341861627247)-- (0.386149427305012,2.461494273050155);
	\begin{scriptsize}
	\draw [fill=ududff] (1.9800574450347974,2.4307471365250777) circle (1.5pt);
	\draw[color=ududff] (1.9587355070890407,2.7329020274754162) node {$v_1$};
	\draw [fill=ududff] (3.4800574450347974,2.4307471365250777) circle (1.5pt);
	\draw[color=ududff] (3.4587355070890407,2.6020274754162) node {$v_2$};
	\draw [fill=ududff] (4.436608636510717,3.0560341861627247) circle (1.5pt);
	\draw[color=ududff] (4.436608636510717,2.8560341861627247) node {$v$};
	\draw [fill=ududff] (6.386149427305012,2.461494273050155) circle (1.5pt);
	\draw[color=ududff] (6.385401237561867,2.799798591305509) node {$v_{k-1}$};
	\draw [fill=ududff] (1.440057445034797,1.4307471365250777) circle (1.5pt);
	\draw [fill=ududff] (2.5000574450347997,1.4307471365250777) circle (1.5pt);
	\draw [fill=ududff] (2.940057445034797,1.4307471365250777) circle (1.5pt);
	\draw [fill=ududff] (4.0000574450347997,1.4307471365250777) circle (1.5pt);
	\draw [fill=ududff] (3.8072408829661954,3.770746434379706) circle (1.5pt);
	\draw [fill=ududff] (5.026091280124835,3.770746434379706) circle (1.5pt);
	\draw [fill=ududff] (5.946149427305012,1.3614942730501547) circle (1.5pt);
	\draw [fill=ududff] (6.886149427305012,1.4314942730501548) circle (1.5pt);
	
	\draw [fill=ududff] (-3.9800574450347974,2.4307471365250777) circle (1.5pt);
	\draw[color=ududff] (-3.9587355070890407,2.7329020274754162) node {$v_1$};
	\draw [fill=ududff] (-2.4800574450347974,2.4307471365250777) circle (1.5pt);
	\draw[color=ududff] (-2.4587355070890407,2.6020274754162) node {$v_2$};
	\draw [fill=ududff] (-1.436608636510717,3.0560341861627247) circle (1.5pt);
	\draw[color=ududff] (-1.436608636510717,2.8560341861627247) node {$v$};
	\draw [fill=ududff] (0.386149427305012,2.461494273050155) circle (1.5pt);
	\draw[color=ududff] (0.385401237561867,2.799798591305509) node {$v_{k-1}$};
	
	\end{scriptsize}
	\end{tikzpicture}
\caption{A member of $\mathcal{S}_{n,k}$.}
\end{figure}

We conjecture that for an integer $k$, $2\le k$, the following statement:
\begin{conjecture}
	Let $k\in \mathbb{N}$. 
\begin{itemize}
	\item For $k\geq 2$, there exists an integer $m_{k}$ such that for every $n\geq m_k$ and every tree $T$ on $n$ vertices: $\sum_{i=1}^{k}\mu_i(L_{P_{n}})\leq \sum_{i=1}^{k}\mu_i(L_T)$;
	\item 
	for $k\geq 3$, there exists an integer $M_{k}$ such that for every $n\geq M_k$ and every tree $T$ on $n$ vertices, we have
	\[
	 \sum_{i=1}^{k}\mu_i(L_T)\leq \sum_{i=1}^{k}\mu_i(L_{S_{n,k}}),
	 \]
	 where $S_{n,k}\in \mathcal{S}_{n,k}$.
\end{itemize}

\end{conjecture}

\subsubsection{Laplacian Matrix: Sum of the Smallest Eigenvalues}

For the smallest nonzero Laplacian eigenvalue of trees we have:
\begin{theorem}\cite{fi}
	Let $n\in \mathbb{N}$ and $T$ be a tree on $n$ vertices, we have
	\[
	\mu_{n-1}(L_{P_{n}})\leq \mu_{n-1}(L_T)\leq \mu_{n-1}(L_{S_{n}})
	\]	
\end{theorem}	

Suppose that $k\in \mathbb{N}$ and $T$ is a tree with $k$ vertices $v_1,\ldots,v_{k}$ of degree $1$. For a center $v$ of $T$ we have  $|d(v,v_i)-d(v,v_j)|\leq 1$ for all $i,j\in[k]$. We denote the set of all such trees of degree $n$ by $\mathcal{P}_{n,k}$.

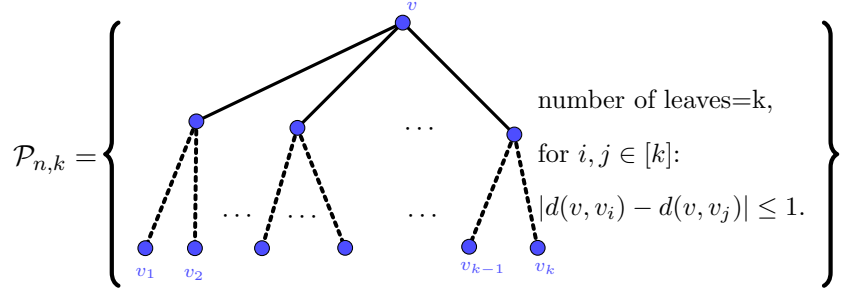
\begin{figure}[H]
	\centering
	\definecolor{ududff}{rgb}{0.30196078431372547,0.30196078431372547,1.}
	\begin{tikzpicture}[scale=1.4,line cap=round,line join=round,>=triangle 45,x=1.0cm,y=1.0cm]
	\draw (5.067231755076255,6.633907315584403) node[anchor=north west] {$\begin{array}{l}
		\text{number of leaves=k,}\\
		\text{for $i,j\in[k]$:}\\
		|d(v,v_i)-d(v,v_j)|\leq 1.
		\end{array}$};
	\draw [line width=1.2pt] (4.,7.)-- (2.040459560278391,6.061608460974373);
	\draw [line width=1.5pt,dotted] (2.040459560278391,6.061608460974373)-- (1.5556895159539323,4.856608636510717);
	\draw [line width=1.5pt,dotted] (2.040459560278391,6.061608460974373)-- (2.0266089875834066,4.856608636510717);
	\draw [line width=1.2pt] (4.,7.)-- (3.,6.);
	\draw [line width=1.5pt,dotted] (3.,6.)-- (2.663735331552695,4.856608636510717);
	\draw [line width=1.5pt,dotted] (3.,6.)-- (3.45322,4.86);
	\draw [line width=1.2pt] (4.,7.)-- (5.059884407785019,5.936953306719512);
	\draw [line width=1.5pt,dotted] (5.059884407785019,5.936953306719512)-- (4.630516654240499,4.870459209205701);
	\draw [line width=1.5pt,dotted] (5.059884407785019,5.936953306719512)-- (5.28149,4.86);
	\draw (3.924137446796283,6.130861324449295) node[anchor=north west] {$\cdots$};
	\draw (3.951838592186252,5.306780971615114) node[anchor=north west] {$\cdots$};
	\draw (2.8091950673674287,5.299826962750222) node[anchor=north west] {$\cdots$};
	\draw (2.1928158599232184,5.305229253530161) node[anchor=north west] {$\cdots$};
	\begin{scriptsize}
	\draw [fill=ududff] (4.,7.) circle (2.pt);
	\draw[color=ududff] (4.090344319136097,7.159682999090536) node {$v$};
	\draw [fill=ududff] (2.040459560278391,6.061608460974373) circle (2.pt);
	\draw [fill=ududff] (1.5556895159539323,4.856608636510717) circle (2.pt);
	\draw[color=ududff] (1.5626148023014226,4.621148900695979) node {$v_1$};
	\draw [fill=ududff] (2.0266089875834066,4.856608636510717) circle (2.pt);
	\draw[color=ududff] (2.019683701235912,4.6072983280009945) node {$v_2$};
	\draw [fill=ududff] (3.,6.) circle (2.pt);
	\draw [fill=ududff] (2.663735331552695,4.856608636510717) circle (2.pt);
	\draw [fill=ududff] (3.45322,4.86) circle (2.pt);
	\draw [fill=ududff] (5.059884407785019,5.936953306719512) circle (2.pt);
	\draw [fill=ududff] (4.630516654240499,4.870459209205701) circle (2.pt);
	\draw[color=ududff] (4.762097094842846,4.664252336865887) node {$v_{k-1}$};
	\draw [fill=ududff] (5.28149,4.86) circle (2.pt);
	\draw[color=ududff] (5.35767172072718,4.66700618780932) node {$v_k$};
	\end{scriptsize}
	\draw[color=black] (.7,5.7) node {{\large $\mathcal{P}_{n,k}=$}};
	\draw [decoration={brace,amplitude=0.5em},decorate,ultra thick,black](1.3,4.5) -- (1.3,7);
	\draw [decoration={brace,amplitude=0.5em},decorate,ultra thick,black](8.,7) -- (8.,4.5);
	
	\end{tikzpicture}
	\caption{The set $\mathcal{P}_{n,k}$}
\end{figure}
For example, for $k=3$, we have the following tree.
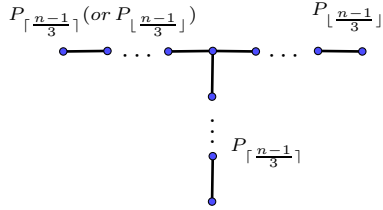
\begin{figure}[H]
	\centering
	\definecolor{ududff}{rgb}{0.30196078431372547,0.30196078431372547,1.}
	\definecolor{cqcqcq}{rgb}{0.7529411764705882,0.7529411764705882,0.7529411764705882}
	\begin{tikzpicture}[scale=1.4,line cap=round,line join=round,>=triangle 45,x=1.0cm,y=1.0cm]
	\draw [line width=1.pt] (4.,5.)-- (4.408907491120748,4.995114363460563);
	\draw [line width=1.pt] (4.,5.)-- (3.9933903102712116,4.565746609916042);
	\draw [line width=1.pt] (3.5778731294216755,4.995114363460563)-- (4.,5.);
	\draw [line width=1.pt] (4.,4.)-- (3.9933903102712116,3.568505375877154);
	\draw [line width=1.pt] (3.,5.)-- (2.5806318953827887,4.995114363460563);
	\draw [line width=1.pt] (5.,5.)-- (5.419999297854619,4.995114363460563);
	\draw (3.046217,5.1) node[anchor=north west] {$\cdots$};
	\draw (4.436155547,5.1) node[anchor=north west] {$\cdots$};
	\draw (3.872733,4.5911) node[anchor=north west] {$\vdots$};
	
	\begin{scriptsize}
	\draw [fill=ududff] (4.,5.) circle (1.pt);
	\draw [fill=ududff] (3.5778731294216755,4.995114363460563) circle (1.pt);
	\draw [fill=ududff] (4.408907491120748,4.995114363460563) circle (1.pt);
	\draw [fill=ududff] (3.9933903102712116,4.565746609916042) circle (1.pt);
	\draw [fill=ududff] (4.,4.) circle (1.pt);
	\draw [fill=ududff] (3.9933903102712116,3.568505375877154) circle (1.pt);
	\draw [fill=ududff] (3.,5.) circle (1.pt);
	\draw [fill=ududff] (2.5806318953827887,4.995114363460563) circle (1.pt);
	\draw [fill=ududff] (5.,5.) circle (1.pt);
	\draw [fill=ududff] (5.419999297854619,4.995114363460563) circle (1.pt);
	\draw[color=black] (4.5197120726806235,4.034010626381) node {$P_{\lceil\frac{n-1}{3}\rceil}$};
	\draw[color=black] (5.28758555317499,5.327528108140193) node {$P_{\lfloor\frac{n-1}{3}\rfloor}$};
	\draw[color=black] (2.95470838641737,5.285976390055239) node {$P_{\lceil\frac{n-1}{3}\rceil}(or\,P_{\lfloor\frac{n-1}{3}\rfloor})$};
	\end{scriptsize}
	\end{tikzpicture}
	\caption{The tree $P_{n,3}$.}
\end{figure}

The sum of the $k$ smallest Laplacian eigenvalues of graphs is a concave function.
We conjecture that for an integer $k$, the following statement for $\displaystyle\sum_{i=n-k+1}^{n}\mu_i(L_{T})=\mu_{n-k+1}(L_{T})+\cdots+\mu_{n}(L_{T})$ :
\begin{conjecture}
	Let $k\in \mathbb{N},k\geq 3$. 
	\begin{itemize}
		\item There exists an integer $M_{k}$ such that for every $n\geq M_{k}$ and every tree $T$ on $n$ vertices: $\displaystyle\sum_{i=n-k+1}^{n}\mu_i(L_{T})\leq \displaystyle\sum_{i=n-k+1}^{n}\mu_i(L_{S_{n}})$;
		\item 
		there exists an integer $m_{k}$ such that for every $n\geq m_k$ and every tree $T$ on $n$ vertices, we have
		\[
\displaystyle\sum_{i=n-k+1}^{n}\mu_i(L_{P_{n,k}})	\leq	\displaystyle\sum_{i=n-k+1}^{}\mu_i(L_T),
		\]
		where $P_{n,k}\in\mathcal{P}_{n,k}$. 	
	\end{itemize}
\end{conjecture}

\subsection{Bounds for Graphs with Given Edges}
For given positive integers $n$ and $e$, the sum of the largest (and smallest) adjacency eigenvalues of graphs of order $n$ and the number of edges $e$ is considered in this section. The set of all simple graphs on $n$ vertices and $e$ edges is denoted by $G(n,e)$.

The extremal graphs for the largest eigenvalue of graphs in  $G(n,e)$ are considered widely (see \cite{bhof},\cite{bs},\cite{oad}, and \cite{ro}).
The following theorem for the largest adjacency eigenvalue of graphs in $G(n,e)$ is conjectured by Brualdi and Hoffman \cite{bhof} and it is proved by Rowlinson \cite{ro}.
\begin{theorem}\cite{ro}
	Let $n,e\in \mathbb{N}$ and $e=\binom{t}{2}+s$ for some integers $t,s$ and $s<t$. If $G$ is a graph of order $n$ with $e$ edges, then
	\[
	\lambda_1(G)\leq \lambda_1(g(n,e)),
	\]	
	where $g(n,e)$ is obtained from the complete graph $K_t$ by adding a new 
	vertex and $s$ new edges and $n-t-1$ isolated vertices.
	
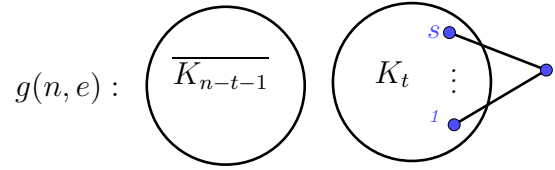
\begin{figure}[H]
		\centering
\definecolor{ududff}{rgb}{0.30196078431372547,0.30196078431372547,1.}
\begin{tikzpicture}[scale=1.5,line cap=round,line join=round,>=triangle 45,x=1.0cm,y=1.0cm]

\draw [line width=1.pt] (4.60241760647923,3.6564944485864985) circle (0.7cm);
\draw (4.1872983280009946,3.9286202659467495) node[anchor=north west] {\Large $K_{t}$};
\draw (4.852125817360252,3.9978731294216723) node[anchor=north west] {$\vdots$};
\draw [line width=1.pt] (5.793964760619201,3.762413393606935)-- (4.9767809716151135,3.2776433492824752);
\draw [line width=1.pt] (5.793964760619201,3.762413393606935)-- (4.93522925353016,4.094827138286564);
\draw [line width=1.pt] (2.952296046028495,3.623907666657089) circle (0.7cm);
\draw (2.4,3.9701719840317033) node[anchor=north west] {\Large $\overline{K_{n-t-1}}$};
\draw (1.,3.8301719840317033) node[anchor=north west] {\Large $g(n,e):$};
\begin{scriptsize}
\draw [fill=ududff] (4.93522925353016,4.094827138286564) circle (1.5pt);
\draw[color=ududff] (4.7967235265803145,4.1033041427190105) node {\large $s$};
\draw [fill=ududff] (4.9767809716151135,3.2776433492824752) circle (1.5pt);
\draw[color=ududff] (4.7967235265803145,3.3538214991048902) node {1};
\draw [fill=ududff] (5.793964760619201,3.762413393606935) circle (1.5pt);
\end{scriptsize}
\end{tikzpicture}
	\end{figure}
\end{theorem}	
We conjecture that for an integer $k$, we have the following statement:
\begin{conjecture}\label{adjmax}
	Let $k,n\in \mathbb{N}$, $2\leq k$, and $G$ be a graph of order $n$ with $e$ edges.  If $t\leq \lfloor \frac{n}{k}\rfloor$, $e= (k-1)\displaystyle\binom{t}{2}+\displaystyle\binom{t-1}{2}+s$, and $s\in [t]$, then there exists an integer $e_t<t$ such that for $s\geq e_t$ we have  $$\sum_{i=1}^{k}\lambda_i(G)\leq (k-1)(t-1)+\lambda_{1}(g(t,\binom{t-1}{2}+s)).$$

	\begin{figure}[H]
		\centering
		\begin{tikzpicture}[scale=1.3,line cap=round,line join=round,>=triangle 45,x=1.0cm,y=1.0cm]
		\draw [line width=1.pt] (5.540977267736963,2.484195593976467) circle (0.8cm);
		\draw (5.2,2.75132158687306) node[anchor=north west] {$K_t$};
		\draw [line width=1.pt] (7.540977267736963,2.484195593976467) circle (0.8cm);
        \draw (6.7,2.75132158687306) node[anchor=north west] {\small $g(t,\binom{t-1}{2}+s)$};
		\draw [line width=1.pt] (1.965316825373121,2.4844051427859) circle (0.8cm);
		\draw (1.597533683716,2.7236204414830905) node[anchor=north west] {$K_t$};
		\draw (3.5224708386417367,2.709769868788106) node[anchor=north west] {$\cdots$};
		\draw (1.9019538333285457,1.7) node[anchor=north west] {$1$};
		\draw (5.2695330671951,1.7) node[anchor=north west] {$k-1$};
		\draw (7.51695330671951,1.7) node[anchor=north west] {$k$};
		\draw (3.536321411336721,1.54) node[anchor=north west] {$\cdots$};
		\end{tikzpicture}
		\caption{The extremal graph of Conjecture \ref{adjmax}.}
	\end{figure}
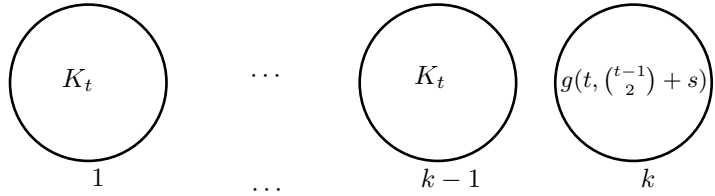
\end{conjecture}

\section*{Acknowledgments}
The author would like to thank Prof. Dariush Kiani for his very careful reading of this paper and for his very helpful advice, and Dr. Sara Saeedi Madani for her helpful comments and suggestions. The author is indebted to Iran National Science Foundation (INSF) for supporting
this research under grant number 95005902. 

\bibliographystyle{amsplain}

\end{document}